\documentclass[12pt]{article}

\usepackage{ amssymb, latexsym, amsfonts, amsmath, lscape, amscd,amsthm, color, epsfig}
\usepackage[a4paper, left=0.8in, right=0.8in, top=1in, bottom=1in]{geometry}

\usepackage[T1]{fontenc}
\usepackage[latin1]{inputenc}
\usepackage[babel=true]{csquotes}
\usepackage{authblk}
\usepackage{setspace}
\usepackage{enumitem}
\usepackage{tikz-network}
\usepackage{tikz}
\usepackage{subfig}

\usetikzlibrary{decorations.pathreplacing}
\newtheorem{theorem}{Theorem}
\numberwithin{theorem}{section}

\numberwithin{conjecture}{section}

\newtheorem{corollary}{Corollary}
\numberwithin{corollary}{section}

\numberwithin{example}{section}

\newtheorem{lemma}{Lemma}
\numberwithin{lemma}{section}

\newtheorem{proposition}{Proposition}
\numberwithin{proposition}{section}

\numberwithin{problem}{section}

\numberwithin{remark}{section}

\newtheorem{claim}{Claim}
\numberwithin{claim}{theorem}

\numberwithin{definition}{section}

\begin{document}
\renewcommand\qedsymbol{$\blacksquare$}
\newcommand\subrel[2]{\mathrel{\mathop{#2}\limits_{#1}}}
\author{Hadeel Al Bazzal$^{1,2}$}
	
\footnotetext[1]{KALMA, Faculty of Sciences, Lebanese University, Baalbek, Lebanon.\\
(hadeel.albazzal@ul.edu.lb)}
\footnotetext[2]{LIB, Universit\'e Bourgogne Europe, Dijon, France.}

\title{On $S$-packing colorings of subcubic graphs}
\maketitle
\begin{abstract}

\noindent Given a sequence \( S = (s_1, s_2, \ldots, s_k) \) of positive integers satisfying \( s_1 \leq s_2 \leq \dots \leq s_k \), an \( S \)-packing coloring of a graph \( G \) is a partition of \( V(G) \) into \( k \) subsets \( V_1, V_2, \dots, V_k \) such that, for each \( 1 \leq i \leq k \), the distance between any two distinct vertices \( x, y \in V_i \) is at least \( s_i + 1 \). Yang and Wu established that every $3$-irregular subcubic graph admits a \( (1,1,3) \)-packing coloring. Later, Mortada and Togni  introduced the concept of  an \( i \)-saturated subcubic graph, defined as a subcubic graph in which every vertex of degree three has at most \( i \) neighbors of degree three  for \( 0 \leq i \leq 3 \). They further demonstrated that all $1$-saturated subcubic graphs are \( (1,1,2) \)-packing colorable. In this paper, we present new concise proofs of these results using a novel tool.
\end{abstract}

\section{Introduction}
\label{intro}
In this paper, we assume all  graphs are simple, meaning they have no loops and multiple  edges. For a graph $G$, we denote by $V(G)$ the set of vertices of $G$ and by $E(G)$ the set of  edges. We denote by \( g(G) \) the girth of \( G \), defined as the length of the shortest cycle in \( G \). For a vertex $v$ in $G$, we denote by $d_{G}(v)$ the number of neighbors of $v$ in $G$. For brevity, we refer to a vertex \( v \) with \( d_G(v) = i \) as an \( i \)-\textit{vertex} in \( G \). We denote by \(\Delta(G)\) the maximum degree of $G$. The distance between two vertices $u$ and $v$ in $G$, denoted by $d_G(u, v)$, is the length of the shortest path between $u$ and $v$ in $G$. For a set $M\subseteq E(G)$, $M$ is said to be a matching in $G$ if no two edges in $M$ share a common vertex. For a set $H\subseteq V(G)$, $H$ is called an independent set in $G$ if no two vertices in $H$ are adjacent. We denote by $G[H]$ the subgraph induced by $H$. A spanning subgraph of a graph $G$ is a subgraph $H$ of $G$ such that $V(H)=V(G)$. A bipartite graph $B$, denoted by $B=X\cup Y$, is a graph where $V(B)$ can be partitioned into two subsets $X$ and $Y$, called partite sets, such that every edge of $B$ joins a vertex of $X$ and a vertex of $Y$.
  
A subcubic graph is a graph in which each vertex has at most three neighbors. Mortada and Togni \cite{MortadaTogni20} introduced the concept of an \( i \)-saturated subcubic graph, defined as a subcubic graph in which every vertex of degree three has at most \( i \) neighbors of degree three, where \( 0 \leq i \leq 3 \). Note that a $3$-irregular subcubic graph is the same as a $0$-saturated subcubic graph.

For a sequence \( S = (s_1, s_2, \ldots, s_k) \) of positive integers satisfying \( s_1 \leq s_2 \leq \dots \leq s_k \), an \( S \)-packing coloring of a graph \( G \) is a partition of \( V(G) \) into \( k \) subsets \( V_1, V_2, \dots, V_k \) such that, for each \( 1 \leq i \leq k \), the distance between any two distinct vertices \( x, y \in V_i \) is at least \( s_i + 1 \). The smallest integer \( k \) for which \( G \) admits a \( (1, 2, \dots, k) \)-packing coloring ($k$-packing coloring) is called the packing chromatic number of \( G \) and is denoted by \( \chi_\rho(G) \). This concept was introduced by Goddard, Hedetniemi, Hedetniemi, Harris, and Rall~\cite{GoddardHedetniemiHarrisRall08} in 2008 under the name broadcast chromatic number. For a graph \( G \), the subdivided graph of \( G \), denoted by \( S(G) \), is obtained by replacing each edge of \( G \) with a path of length two.  

Numerous studies have focused on establishing bounds for \( \chi_\rho(G) \) and \( \chi_\rho(S(G)) \) (see \cite{BaloghKostochkaLiu1, BresarFerme03, BresarKlavzarRallWash05, BresarKlavzarRallWash06, GastineauTogni12}). In particular, Gastineau and Togni~\cite{GastineauTogni12} posed the question of whether the inequality \( \chi_\rho(S(G)) \leq 5 \) holds for any subcubic graph \( G \), a conjecture later formalized by Bre\v{s}ar, Klav\v{z}ar, Rall, and Wash~\cite{BresarKlavzarRallWash05}. Gastineau and Togni~\cite{GastineauTogni12} demonstrated that for a subcubic graph \( G \) to satisfy \( \chi_\rho(S(G)) \leq 5 \), it is sufficient for \(G\) to be \((1,1,2,2)\)-packing colorable. Several papers ~\cite{LiuLiuRolekYu17, MortadaTogni20} have verified that certain subclasses of subcubic graphs are \((1,1,2,2)\)-packing colorable. More recently, Bre{\v s}ar, Kuenzel, and Rall \cite{Bresar1} proved that every claw-free cubic graph is \( (1,1,2,2) \)-packing colorable.  

 Many other $S$-packing colorings have been studied. In particular,  Gastineau and Togni~\cite{GastineauTogni12} proved that every subcubic graph is $(1, 1, 2, 2, 2)$-packing colorable. Tarhini and Togni \cite{i} proved that every cubic Halin graph is $(1, 1, 2, 3)$-packing colorable. Mortada and Togni~\cite{2-saturated 2, i-saturated} established many results regarding $i$-saturated subcubic graphs. Recently, Liu, Zhang, and Zhang \cite{LiuZhang} proved that every subcubic graph is (1, 1, 2, 2, 3)-packing colorable, and hence that the packing chromatic number of the subdivided graph of any subcubic graph is at most six.

 Yang and Wu \cite{a} answered a question posed by Gastineau and Togni \cite{GastineauTogni12}, confirming that every $0$-saturated subcubic graph is \( (1,1,3) \)-packing colorable. This result has improved the earlier work of Gastineau and Togni \cite{GastineauTogni12}, who had shown that every $0$-saturated subcubic graph is \( (1,1,2) \)-packing colorable. More precisely, Yang and Wu \cite{a} demonstrated that every $0$-saturated subcubic graph admits a good \( (1,1,3) \)-packing coloring, where a \( (1,1,3) \)-packing coloring of \( G \) is said to be a good \( (1,1,3) \)-packing coloring if no connected component of  the subgraph induced by the vertices of degree two in $G$ is isomorphic to a path of length two, \( P = uvw \), with the vertex \( v \) colored by the color three. To prove this result, Yang and Wu considered the $0$-saturated subcubic graph of minimum order that does not admit such a coloring. Then, they  proved that every vertex of degree two is adjacent to at most one vertex of degree two, \( G \) is 2-connected, and \( g(G) \) is at least $3$, then at least $4$, $5$, $6$, and finally at least $7$. These results allowed the authors to define the distance between edges whose ends are both of degree two in \( G \). Ultimately, they proved the existence of a good \( (1,1,3) \)-packing coloring for \( G \). Mortada and Togni~\cite{MortadaTogni20} demonstrated that all $1$-saturated subcubic graphs are \((1,1,2)\)-packing colorable. Their proof technique involves considering a $1$-saturated subcubic graph \( G \) of minimum order that does not admit such a coloring. Then, they select an independent set in  \( G \) that maximizes, among all independent sets, a linear combination of the number of vertices of degree three and the number of vertices of degree two. By analyzing such an independent set, they were able to determine the distance between a number of vertices in \( G \) and establish the existence of a \( (1,1,2) \)-packing coloring for \( G \). Using the same method, Mortada \cite{Mortada} provided a simpler proof for the result established by Yang and Wu \cite{a} concerning the $0$-saturated subcubic graphs.

 In this paper, we present concise proofs that every $0$-saturated subcubic graph is \((1,1,3)\)-packing colorable and that every $1$-saturated subcubic graph is \((1,1,2)\)-packing colorable. Our approach depends on the spanning bipartite subgraph \( B \) of \( G \) with the maximum number of edges. This subgraph effectively organizes the structure of the subcubic graph, facilitating the partitioning of \( V(G) \) and enabling the achievement of the desired colorings. Notably, this technique is introduced for the first time in the context of $S$-packing colorings of subcubic graphs.

\section{About $0$-Saturated and $1$-Saturated Subcubic Graphs}
\label{sec:1}

\noindent Let us begin by introducing a key property of the spanning bipartite subgraph of a graph \( G \) with the maximum number of edges.

\begin{proposition}\label{p1}
  Let \( G \) be a graph. Let \( B = X \cup Y \) be a spanning bipartite subgraph of \( G \) with the maximum number of edges. Then \( d_B(v) \geq \frac{1}{2} d_G(v) \) for any \( v \in G \).
\end{proposition}
\begin{proof}
 Suppose, to the contrary, that there exists a vertex \( x \in G \) such that \( d_B(x) < \frac{1}{2} d_G(x) \). Without loss of generality, suppose that \( x \in X \). Now,  define \( B' = (X \setminus \{x\}) \cup (Y \cup \{x\}) \) as a spanning bipartite subgraph of \( G \), where $E(B')$ consists of all edges of \( G \) that have one endpoint in \( X' \) and the other in \( Y' \). Thus, we have \( |E(B')| = |E(B)| + d_{G[X]}(x) - d_B(x) \). Since \(  d_{G[X]}(x)=d_G(x)- d_B(x) \), \( |E(B')| = |E(B)| + d_G(x) - 2d_B(x) \). Therefore, \( |E(B')| > |E(B)|  \) as \( d_B(x) < \frac{1}{2} d_G(x) \), a contradiction.\end{proof}

\noindent By specializing the graph in Proposition \ref{p1} to a subcubic graph, we obtain the following corollary.

\begin{corollary}\label{p2}
  Let \( G \) be a subcubic graph. Let \( B = X \cup Y \) be a spanning bipartite subgraph of \( G \) with the maximum number of edges. Then \( \Delta(G[X]) \leq 1 \) and \( \Delta(G[Y]) \leq 1 \).
\end{corollary}
\begin{proof}
 Let \( x \in X \). We have \( d_{G[X]}(x)= d_G(x) - d_B(x) \). By Proposition \ref{p1}, we get \( d_{G[X]}(x) \leq \frac{1}{2} d_G(x) \). Since \( G \) is a subcubic graph, \( d_{G[X]}(x) \leq 1 \). Similarly, we can show that \( \Delta(G[Y]) \leq 1 \).\end{proof}

\noindent For any spanning bipartite subgraph $B=X \cup Y$ of a subcubic graph \textit{G} with the maximum number of edges, let $E_X$ (respectively $E_Y$) denote the set of edges of $G[X]$ (respectively $G[Y]$). By Corollary \ref{p2}, each \( E_X \) and \( E_Y \) is a  matching in \( G \). We call each end of an edge in $E_X \cup E_Y$ a bad vertex. By Proposition \ref{p1}, each bad vertex has degree at least two in $G$. Two bad vertices in $G$ are said to be \textit{unlinked} if they belong to distinct edges in \( E_X \cup E_Y \).\\\

\noindent Our approach establishing that a subclass of a subcubic graph \( G \) is \( (1,1,k) \)-packing colorable follows this strategy. We start by selecting a spanning bipartite subgraph \( B = X \cup Y \) of \( G \) with the maximum number of edges. Then, we color exactly one end of each edge in \( E_X \cup E_Y \) with \( k \). By Corollary \ref{p2}, and since each edge has a colored end, the subgraph of $G$ induced by the uncolored vertices is a bipartite graph. This establishes that \( G \) is \( (1,1,k) \)-packing colorable.\\

\noindent We will now present two lemmas that are essential for proving the main theorems of this paper.

\begin{lemma}\label{l1}
   Let \( G \) be a subcubic graph, and let \( B = X \cup Y \) be a spanning bipartite subgraph of \( G \) with the maximum number of edges. If \( x \in X\) and \( y \in Y\) are two adjacent bad vertices, then both $x$ and $y$ are $3$-vertices in $G$.
\end{lemma}
\begin{proof}
    Suppose, to the contrary, that either $x$ or $y$ is a $2$-vertex in $G$. Without loss of generality, assume that $y$ is a $2$-vertex in $G$. Then, define $B'=X'\cup Y'$ to be a spanning bipartite subgraph of $G$, where $X'=(X\setminus\{x\})\cup \{y\}$ and $Y'=(Y\setminus\{y\})\cup \{x\}$. The edge set $E(B')$ consists of all edges of \( G \) that have one endpoint in \( X' \) and the other in \( Y' \). Thus, we have $|E(B')| = |E(B)|+d_{G[X]}(x)+d_{G[Y]}(y)-d_{G[Y']}(x)-d_{G[X']}(y)$. Clearly, $d_{G[X]}(x)=d_{G[Y]}(y)=1$. Since $x$ has two neighbors in $X'$ and \( G \) is a subcubic graph, $d_{G[Y']}(x)\leq 1$. Additionally, $d_{G[X']}(y)=0$ because $y$ is a $2$-vertex in $G$, and $y$ already has two neighbors in $Y'$. Thus \( |E(B')| > |E(B)| \),  a contradiction.
\end{proof}

 \begin{lemma}\label{l2}
Let \( G \) be a subcubic graph, and let \( B = X \cup Y \) be a spanning bipartite subgraph of \( G \) with the maximum number of edges. Suppose \( u, v \in V(G) \) are bad unlinked vertices belonging to the same partite set of \( B \) and having a common neighbor \( z \in V(G) \). Then, \( u \) or \( v \) is a $3$-vertex in \( G \). Furthermore, if either \( u \) or \( v \) is a $2$-vertex in \( G \), then \( z \) must be a $3$-vertex in \( G \). \end{lemma}
\begin{proof}
  Let \( u \) and \( v \) be two bad unlinked $2$-vertices in \( G \). Then, there exist  \( u_1,v_1\in V(G) \) such that $uu_1,vv_1 \in E_X\cup E_Y$. Suppose, without loss of generality, that \( u, v \in X \) and that they have a common neighbor \( z \). By Corollary \ref{p2}, it follows that \( z \in Y \).
 
  \noindent For the first part of this lemma, suppose to the contrary that both $u$ and $v$ are $2$-vertices in $G$. Define \( B' = X' \cup Y' \) to be a spanning bipartite subgraph of \( G \), where \( X' = (X\setminus \{u, v\}) \cup \{z\} \) and \( Y' = (Y \setminus \{z\})\cup \{u,v\}\). The edge set $E(B')$ consists of all edges of \( G \) that have one endpoint in \( X' \) and the other in \( Y' \). Thus, we have $ |E(B')| = |E(B)| + d_{G[Y]}(z) + d_{G[X]}(u) + d_{G[X]}(v) - d_{G[X']}(z) - d_{G[Y']}(u)- d_{G[Y']}(v)$. As $u$ and $v$ are bad vertices in $G$, $d_{G[X]}(u)=d_{G[X]}(v)=1$. Since $z$ has two neighbors in $Y'$ and \( G \) is a subcubic graph,  we have $d_{G[X']}(z)\leq 1$. Moreover, $d_{G[Y']}(u)=d_{G[Y']}(v)=0$ because both $u$ and $v$ are $2$-vertices in $G$ and already have two neighbors in $X'$. Thus \( |E(B')| > |E(B)| \),  a contradiction.
  
 \noindent For the second part of this lemma, assume without loss of generality that $u$ is a $2$-vertex in $G$ and suppose, for the sake of contradiction, that $z$ is also a $2$-vertex in $G$. Now, define $B'=X'\cup Y'$ as a spanning bipartite subgraph of $G$, where $X'=X\setminus\{u\}$ and $Y'=Y\cup \{u\}$. Then, we obtain $|E(B')| = |(E(B)\setminus \{uz\})\cup\{uu_1\}|=|E(B)|$. Therefore, \( B' \) is a spanning bipartite subgraph of \( G \) with the maximum number of edges. However, under this new bipartition, \( v \in X' \) and \( z \in Y' \) are two adjacent bad vertices, where $z$ is a $2$-vertex in $G$, a contradiction by Lemma \ref{l1}.\end{proof}

\noindent The following lemma establishes the distance between any two bad unlinked $2$-vertices in \( G \).

\begin{lemma}\label{l3}
Let \( G \) be a subcubic graph, and let \( B = X \cup Y \) be a spanning bipartite subgraph of \( G \) with the maximum number of edges. Then, the distance between any two bad unlinked $2$-vertices in \( G \) is at least 3.
 \end{lemma}
\begin{proof}
  Let \( u \) and \( v \) be two bad unlinked $2$-vertices in \( G \). Then, there exist  \( u_1,v_1\in V(G) \) such that $uu_1,vv_1 \in E_X\cup E_Y$. First, suppose that \( u \) and \( v \) belong to the same partite set of \( B \), either \( X \) or \( Y \). As \( u \) and \( v \) are unlinked, \( d_G(u, v)> 1 \). Furthermore, by Corollary \ref{p2} and by Lemma~\ref{l2}, \( u \) and \( v \) cannot share a common neighbor, which implies that \( d_G(u, v) \geq 3 \). Next, consider the case where \( u \) and \( v \) belong to distinct partite sets of \( B \). By Lemma~\ref{l1}, we have \( uv, uv_1, u_1v \notin E(G) \). As a result, \( d_G(u, v) \geq 3 \).
 \end{proof}

\noindent We are now ready to present the proofs of our main theorems.

\begin{theorem}\label{t3}
    Every $0$-saturated subcubic graph is $(1,1,3)$-packing colorable.
\end{theorem}
\begin{proof}
    Let \textit{G} be a $0$-saturated subcubic graph. Let \( B = X \cup Y \) be a spanning bipartite subgraph of \textit{G} with the maximum number of edges. By Corollary \ref{p2}, \( E_X \) and \( E_Y \) are two matchings in \( G \). By Proposition \ref{p1}, each bad vertex has degree at least two in \( G \). Therefore, since \( G \) is a $0$-saturated subcubic graph, at least one end of each edge in \( E_X \cup E_Y \) must be a $2$-vertex in \( G \). 
    \begin{claim}\label{c3}
        The distance between any two bad unlinked $2$-vertices in $G$ is at least 4.
    \end{claim}
    \noindent \textit{Proof.}
          Let \( u \) and \( v \) be two bad unlinked $2$-vertices in \( G \). Then, there exist \( u_1,v_1\in V(G) \) such that $uu_1,vv_1 \in E_X\cup E_Y$. By Lemma \ref{l3}, \( d_G(u, v) \geq 3 \). Suppose that \( d_G(u, v) = 3 \).
         
         \noindent First, assume that \( u \) and \( v \) are in the same partite set. Then, either \( u \) and \( v_1 \) or \( u_1 \) and \( v \) have a common neighbor in \( G \). Without loss of generality, suppose that \( u \) and \( v_1 \) have a common neighbor \( z \in V(G) \). By Corollary \ref{p2}, \( z \in Y\). By Lemma \ref{l2}, and since \( u \) is a $2$-vertex in \( G \), both \( v_1 \) and \( z \) are $3$-vertices in \( G \), which is a contradiction because \( G \) is a $0$-saturated subcubic graph.

         \noindent Now, consider the case where \( u \) and \( v \) are in distinct partite sets. Without loss of generality, assume that \( u \in X \) and \( v \in Y \). By Lemma \ref{l1}, and since \( G \) is a $0$-saturated subcubic graph, \( u_1v_1 \notin E(G) \). Then, there exist \( x \in X \) and \( y \in Y \) such that \( uy, yx, xv \in E(G) \).  Now, define \( B' = X' \cup Y' \) to be a spanning bipartite subgraph of $G$, where \( X' = X \cup \{v\} \) and \( Y' = Y \setminus \{v\} \). We have $|E(B')| = |(E(B)\setminus \{xv\})\cup \{vv_1\}|= |E(B)|$. Then, \( B' \) is a spanning bipartite subgraph of \( G \) with the maximum number of edges. In this new bipartition, the two bad unlinked vertices \( u, x \in X' \) have a common neighbor \( y \in Y' \). Then, by Lemma \ref{l2}, and since \( u \) is a $2$-vertex in \( G \), it follows that both \( x \) and \( y \) must be $3$-vertices in \( G \), a contradiction as \( G \) is a $0$-saturated subcubic graph.\hfill $\Box$\par
        \noindent Now, define \( T \subseteq V(G) \) as a set that contains one bad $2$-vertex of each edge in \( E_X \cup E_Y \). By Claim \ref{c3}, we can color each vertex in \( T \) with 3. Moreover, by Corollary \ref{p2} and the construction of \( T \), the subgraph of \( G \) induced by the remaining uncolored vertices is a bipartite graph. This implies that \( G \) is \( (1,1,3) \)-packing colorable. \end{proof}

\begin{theorem}\label{t2}
    Every $1$-saturated subcubic graph is $(1,1,2)$-packing colorable.
\end{theorem}
\begin{proof}
     Let \textit{G} be a $1$-saturated subcubic graph. Consider a spanning bipartite subgraph $B = X \cup Y$ of \textit{G} with the maximum number of edges. By Corollary \ref{p2}, \( E_X \) and \( E_Y \) are two matchings in \( G \). By Proposition \ref{p1}, each bad vertex has degree at least two in $G$. Consequently, we define \( E_1(B) \) to contain each edge in  \( E_X \cup E_Y \) where at least one of its ends is a $2$-vertex in $G$, and \( E_2(B) = (E_X \cup E_Y) \setminus E_1(B) \); that is, both ends of any edge in $E_2(B)$ are $3$-vertices in $G$. Assume $B$ is chosen such that $|E_1(B)|\geq |E_1(B')|$ for each spanning bipartite subgraphs $B'$ of \textit{G} with the maximum number of edges. Now, define $T$ as the set containing exactly one bad $2$-vertex of each edge in $E_1(B)$ and exactly one end of each edge in $E_2(B)$.
 \begin{claim}\label{c2}
      Let $u,v\in T$, then $d_G(u, v) \geq 3$.
 \end{claim} 
  \noindent \textit{Proof.}
  \noindent Let $u,v\in T$. By the definition of \( T \), $u$ and $v$ are unlinked. Then, there exist \( u_1, v_1 \in V(G) \) such that \( uu_1, vv_1 \in E_X \cup E_Y \). Suppose that \( u \) and \( v \) belong to different partite sets. We have \( uv, uv_1, vu_1 \notin E(G) \). Indeed, suppose to the contrary and without loss of generality that \( uv \in E(G) \). Then, by Lemma \ref{l1}, both $u$ and $v$ are $3$-vertices in $G$. Consequently, by the definition of \( T \), both $u_1$ and $v_1$ are $3$-vertices in $G$, a contradiction as \( G \) is a $1$-saturated subcubic graph. Suppose now that both \( u \) and \( v \) are in the same partite set, either \( X \) or \( Y \). Without loss of generality, assume \( u, v \in X \). By the definition of \( T \), \( d_G(u, v)> 1 \). Assume \( d_G(u, v) = 2 \). Then, by Corollary \ref{p2}, there exists a vertex \( y \in Y \) such that \( uy, yv \in E(G) \). By Lemma \ref{l2}, either \( u \) or \( v \) is a $3$-vertex in $G$. Suppose, without loss of generality, that $u$ is a $3$-vertex in $G$. By the definition of \( T \), $u_1$ must be a $3$-vertex in $G$. Since \( G \) is a $1$-saturated subcubic graph, $y$ is a $2$-vertex in $G$. Thus, by Lemma \ref{l2}, \( v \) is a $3$-vertex in $G$. By the definition of \( T \), $v_1$ must be a $3$-vertex in $G$. Now, let $u'$ (respectively $v'$) be the third neighbor of $u$ (respectively $v$) distinct from both $u_1$ (respectively $v_1$) and $y$. By Corollary \ref{p2}, \(u', v'\in Y \). Since $G$ is a $1$-saturated subcubic graph, both $u'$ and $v'$ are $2$-vertices in $G$. Now, define $B'=X'\cup Y'$ as a spanning bipartite subgraph of $G$, where $X'=(X\setminus\{u,v\})\cup \{y\}$ and $Y'=(Y\setminus\{y\})\cup \{u,v\}$. We have $ |E(B')| = |(E(B)\setminus \{uu',vv'\})\cup \{uu_1,vv_1\}|= |E(B)|$. Therefore, \( B' \) is a spanning bipartite subgraph of \( G \) with the maximum number of edges. However, $E_1(B')=E_1(B) \cup \{uu_1, vv_1\}$, which contradicts the maximality of $E_1(B)$.\hfill $\Box$\par
\noindent By Claim \ref{c2}, we can color each vertex in \( T \) with 2. Moreover, by the construction of \( T \) and Corollary \ref{p2}, the subgraph of \( G \) induced by the remaining uncolored vertices is a bipartite graph. Therefore, \( G \) is \( (1,1,2) \)-packing colorable.
\end{proof}

\end{document}